\documentclass[a4paper,10pt,twoside]{amsart}

\usepackage[latin1]{inputenc}
\usepackage{amsmath, amsfonts, amssymb,amsthm}
\usepackage[mathscr]{eucal} 
\usepackage[pdftex]{graphicx}
\usepackage{color}
\usepackage{multicol}
\usepackage[T1]{fontenc}
\usepackage[sc]{mathpazo}
\usepackage[all]{xy}
\usepackage{hyperref}
\usepackage[flenq]{nccmath}

\usepackage {anysize}
\marginsize{2cm}{2cm}{2cm}{3cm}

\usepackage{enumerate}

\usepackage[pagewise]{lineno}

\definecolor{alert}{rgb}{0.8,0,0}

\newcommand{\R}{\mathbb{R}}

\newcommand{\s}{\mathbb{S}}
\newcommand{\h}{\mathbb{H}}

\newcommand{\M}{\mathbb{M}}

\newcommand{\grad}{\mathrm{grad}}

\newcommand{\sfr}{\mathbb{S}_f\times\R}

\newcommand{\rf}{\mathbb{R}^2_f\times\R}
\newcommand{\mrf}{\mathbb{M}(\kappa)_f\times I}

\DeclareMathOperator{\cosec}{cosec}

\newtheorem{theorem}{Theorem}[section]
\newtheorem{proposition}[theorem]{Proposition}

\newtheorem{lemma}[theorem]{Lemma}
\newtheorem{claim}[theorem]{Claim}

\theoremstyle{definition}
  \newtheorem{definition}[theorem]{Definition}

\theoremstyle{remark}
\newtheorem{remark}[theorem]{Remark}

\numberwithin{equation}{section}

\title[Totally umbilical surfaces in $\mrf$]{On Totally umbilical surfaces in the warped product $\mrf$}

\author{Ady Cambraia Jr.}
\address{Departamento de Matem\'atica, Universidade Federal de Viçosa, Minas Gerais, Brazil, 36570-900}
\email{ady.cambraia@ufv.br}

\author{Abigail Folha}
\address{Departamento de Geometria, Universidade Federal Fluminense, Rio de Janeiro, Brazil, 24210-201}
\email{abigailfolha@vm.uff.br}

\author{Carlos Pe\~{n}afiel}
\address{ Instituto de Matem\'{a}tica, Universidade Federal de Rio de Janeiro, Rio de Janeiro, Brazil, 21941-909}
\email{penafiel@im.ufrj.br}

\thanks{This work was partially supported by Fapemig, Funda\c{c}\~ao de Amparo \`a Pesquisa do Estado de Minas Gerais - Brazil.}


\begin{document}

\begin{abstract}
In this article we classify the totally umbilical surfaces which are immersed into a wide class of Riemannian manifolds having a structure of warped product, more precisely, we show that a totally umbilical surface immersed into the warped product $\mrf$ (here, $\M(\kappa)$ denotes the 2-dimensional space form, having constant curvature $\kappa$, $I$ an interval and $f$ the warping function) is invariant by a one-parameter group of isometries of the ambient space. We also find the first integral of the ordinary differential equation that the profile curve satisfies (we mean, the curve which generates a invariant totally umbilical surface). Moreover, we construct explicit examples of totally umbilical surfaces, invariant by one-parameter group of isometries of the ambient space, by considering certain non-trivial warping function.
 \end{abstract}

\maketitle

\section{Introduction}
\label{intro}

The classification of totally umbilical surface inside the 3-dimensional space forms is well known. In $\R^3$ such surfaces are planes and round spheres. In $\s^3$ they are round spheres. In $\h^3$ they are totally geodesic planes and their equidistants, horospheres and round spheres. All such surfaces have constant main curvatures. Recall the following definition.
\begin{definition}\label{defiumb}
Let us suppose that $\Sigma\looparrowright M$ is a smooth surface, which is isometrically immersed into the manifold $M$. We say that $\Sigma$ is totally umbilical, when there exists a function $\varrho\in C^{\infty}(\Sigma)$ such that $SX =-\overline{\nabla}_XN =\varrho X$ for any $X\in\chi(\Sigma)$, where $N$ is a smooth unit normal vector field to the immersion, $\overline{\nabla}$ the Riemannian connection of $M$ and $S$ its associated Weingarten operator.
\end{definition}


 It is well known that the space forms can be seen as warped product spaces. More precisely, if  we consider the connected, simply connected, complete, 2-dimensional space form $\M(\kappa)$ having constant curvature $\kappa$ and a smooth real function (the warping function) $f:I\subset\R \to \R$, here $I$ denotes an open interval, the warped product space $\mrf$ is the product manifold $\M(\kappa)\times I$ endowed with metric $\overline{g}=e^{2f}g_2+g_1$, where $g_2$ denotes the metric of $\M(\kappa)$ and $g_1$ denotes the metric of $I$. Under this framework for $\kappa=0$ and $I=\R$, the warping function $f\equiv0$ give rise to the euclidean space $\R^3$, whereas a model for the hyperbolic 3-dimensional space $\h^3$ is given if we consider the warping function $f(t)=t$. Notice that even for "close" warping functions, the geometry of the totally umbilical surfaces are dramatically different, as a matter of fact the geometry of such surfaces depends deeply on the warping function. 
 


The characterization of totally umbilical surfaces immersed into warped product manifolds has attracted much attention recently. In \cite{brendle}, Simon Brendle shows that surfaces having constant mean curvature, immersed into certain warped product spaces are actually, totally umbilical surfaces (this result can be apply to the deSitter-Schwarzschild and Reissner-Nordstrom manifolds). More recently, Shanze Gao and Hui Ma characterized immersed totally umbilical surfaces into certain warped product spaces with some additional condition on the mean curvature, \cite{gao2019characterizations}.  For the simply-connected  homogeneous 3-dimensional manifolds having 4-dimensional isometry group, usually labelled by $E(\kappa,\tau)$ for real numbers $\kappa$ and $\tau$ satisfying $\kappa\neq4\tau^2$, such classification was made in two steps. First, for the Heisenberg group $E(\kappa=0,\tau)$,  \cite{sanini1997gauss}, and then  for the remains of the $E(\kappa,\tau)$ as well as for the $Sol$ geometry which has a 3-dimensional isometry group,  \cite{souam2006totally}.

Following with the study of the geometric properties of the totally umbilical surfaces immersed into 3-dimensional manifolds, it is a natural question for looking for totally umbilical surfaces inside the warped products $\M(\kappa)_f\times I$.  In this article, we give the classification of such totally umbilical surfaces which are immersed into $\M(\kappa)_f\times I$ (see Theorem \ref{t1}). More precisely, our main result establishes:

\noindent {\bf Main Theorem.} {\it Let $\Sigma$ be a totally umbilical surface immersed in the warped product  $\mrf$. Then $\Sigma$
is contained in a totally umbilical surface which is invariant by a one-parameter group of isometries of $\mrf$.
}

The article is organized as follow: In section \ref{preliminares} we give some geometric properties of the warped product $\mrf$ as well as we fix some notations. In section \ref{examples}, we give the description of the totally umbilical surfaces invariant by isometries of the ambient space. In section \ref{class} we stablish the classification theorem.  In section \ref{ntex} we construct explicit examples of totally umbilical surfaces, invariant by isometries of the ambient space. We ended the article with an Appendix where we study the geometry of some curves and some vertical cylinders.
   
\section{Preliminares}\label{preliminares} 


\subsection{The ambient space}
Let $\mathbb{M}(\kappa)$ be the connected simply connected  complete  two-dimensional space form having constant curvature $\kappa, \kappa\in \{-1,0,1\}$. So,  given a  $\kappa$, we have
\begin{itemize}
\item[(a)] $\mathbb{M}(-1)=\h^2$, where $\h^2$ denotes the hyperbolic space.
\item[(b)] $\mathbb{M}(0)=\R^2$, where $\R^2$ denotes the euclidean space.
\item[(c)] $\mathbb{M}(1)=\s^2$, where $\s^2$ denotes the euclidean unit sphere.
\end{itemize}

The space $\mathbb{M}(\kappa)$ is endowed with the metric $g_2=\lambda^2(dx^2+dy^2)$, where
\begin{linenomath*}$$\lambda=\left\{
\begin{array}{cc}
\dfrac{2}{1+\kappa(x^2+y^2)},& \ \ \ \textnormal{if $\kappa\neq0$}\\ 
 1,& \ \ \ \textnormal{if $\kappa=0$}
\end{array}
\right..
$$
\end{linenomath*}

Let us consider the 3-dimensional Riemannian product $\mathbb{M}(\kappa) \times I$, here $I$ denotes an open interval, and the canonical projections in the first and second factor 
\begin{linenomath*}$$\pi_1:\mathbb{M}(\kappa) \times I \to \mathbb{M}(\kappa) \hspace{.3cm} \textnormal{and} \hspace{.3cm} \pi_2:\mathbb{M}(\kappa) \times I \to I$$ 
\end{linenomath*}
defined by $\pi_1(p,t)=p$ and $\pi_2(p,t)=t$, respectively. 

Let $f: I\to \R$ be a smooth real function.  We consider the warped product $\mrf=(\M(\kappa)\times I,\overline{g})$ endowed with the metric
\begin{linenomath*}
\begin{equation*}
{g}=e^{2f}g_2+g_1,
\end{equation*}
\end{linenomath*}
where $g_1:=dt^2.$ Thus, for $v\in T_{(p,t)}\mrf$, a tangent vector to $\mrf$ at $(p,t)$, we have
\begin{linenomath*}
$$\overline{g}(v,v)=e^{2f(t)}g_2(d\pi_1(v),d\pi_1(v))+g_1(d\pi_2(v),d\pi_2(v)).$$
\end{linenomath*}


We denote by $\nabla^1$, $\nabla^2$, $\overline{\nabla}$ the Levi-Civita connection of $\R$, $\M(\kappa)$ and $\mrf$, respectively. The tensor curvature of $\mrf$  and $\M(\kappa)$ are denoted by $\overline{R}$ and $R^2$, respectively. Given vector fields $A,B,C\in\chi(\mrf)$, we define the Riemannian curvature tensor by
\begin{linenomath*}
 $$\overline{R}(A,B)C=\overline{\nabla}_B\overline{\nabla}_AC-\overline{\nabla}_A\overline{\nabla}_BC +\overline{\nabla}_{[A,B]}C.$$
\end{linenomath*}

In order to give a relation between vector fields on $\mathbb{M}, I$ and on $\mrf$, we need some definitions.

\begin{definition}
Let $\phi:N_1\ \to\ N_2$ be a smooth mapping between two differentiable manifolds. Vector fields $X\in\chi(N_1)$  and $Y\in\chi(N_2)$ are $\phi$-related if
\begin{linenomath*}
$$d \phi(X(p))=Y(\phi(p)), \ \ \textnormal{for all }\ \ p\in\ N_1.$$
\end{linenomath*}
We recall that:
\begin{itemize}
\item The  {\emph lift} of a vector field $X\in\chi(\M(\kappa))$ is the unique vector field  $\overline{X}\in\chi(\mrf)$ that is $\pi_1$-related to $X$ and $\pi_2$-related to zero vector field on $\chi(\mathbb{R})$, where $\pi_1$ and $\pi_2$ are the canonical projections. We denote by $\mathfrak{L}(\M(\kappa))\subset\chi(\mrf)$ the set of all such lifts.
\item Similarly,  $\mathfrak{L}(I)\subset\chi(\mrf)$ denotes  the set of all lifts of vector fields $V\in\chi(I)$ to $\mrf$, that is, the vector fields $\overline{V}\in\chi(\mrf)$ which are $\pi_2$-related to vector field $V\in\chi(I)$ and $\pi_1$-related to zero vector field on $\chi(\mathbb{M}(\kappa))$.
\item Let $\xi=\textit{lift}(\partial_t)$, where $\partial_t$ denotes the unit tangent vector field to the real line $\R$ which points to the positive direction. 
\end{itemize}
\end{definition}

\begin{remark}
When there is no confusion, by abuse of notation, we will use the same notation for a vector field and for its lifting. If necessary, we will use the over-bar to emphasize the lift of a vector field. 
\end{remark}

Now, we summarize some properties of the warped product $\mrf$.
%

\begin{lemma}\label{l1}
Let $f:I \to \R$ be a smooth real function, $\M(\kappa)$ the two-dimensional space form having constant Gaussian curvature $\kappa$ and consider the warped product $\mrf$ whose vector field tangent to the leaves $I$ is denoted by $\xi$. If $X,Y,Z\in\mathfrak{L}(\M(\kappa))$ and $V,W,U\in\mathfrak{L}(I)$, then
\begin{enumerate}
\item $\overline{\nabla}_VW\in\mathfrak{L}(I)$ is the lift of $\nabla^1_VW$ on $I$, that is, $\overline{\nabla}_VW=\overline{\nabla^1_VW}$.  And $\overline{\nabla}_XV=\overline{\nabla}_VX=V(f)X$.
\item $\overline{\nabla}_XY=\overline{\nabla^2_XY}-\overline{g}(X,Y)f^\prime(t)\xi$.
\item $[X,V]=0$, $[X,Y]=\overline{[X,Y]^2}$ and $[V,W]=\overline{[V,W]^1}$, where $[\cdot,\cdot]$ denotes the Lie bracket.
\item $\overline{R}(V,W)U=\overline{R^1(V,W)U}=0$. 
\item $\overline{R}(X,Y)Z=\kappa[g_2(X,Z)Y-g_2(Y,Z)X]+(f^\prime)^2[\overline{g}(Y,Z)X- \overline{g}(X,Z)Y]$.
\item $\overline{R}(V,X)W =-\overline{g}(V,W)(f^{\prime\prime}+(f^\prime)^2)X.$
\item $\overline{R}(V,W)X=0$ and $\overline{R}(X,Y)V=0.$
\item $\overline{R}(V,X)Y=\overline{g}(X,Y)[-V(f)f^\prime+V(f^\prime)]\xi.$
\end{enumerate}
\end{lemma}

\subsection{Isometric immersions}

Let $\varphi:\Sigma \looparrowright\mrf$ be an immersion from a complete, orientable surface into the warped product $\mrf$. For simplicity, we treat properties of the immersion $\varphi$ as those of $\Sigma$ and denote merely by $\Sigma$ the image $\varphi(\Sigma)$.  We denote by $g$ the pulled back metric $\overline{g}$  on $\Sigma$ and by $\nabla$ the Levi-Civita connection associated to $g$. 

Let $N$ be a unit normal vector field on $\Sigma$ (from now on, we fix the orientation $N$ of $\Sigma$), denote by $S$ the shape operator and by $B$ the second fundamental form of $\Sigma$. Then, we have $\overline{g}(B(X,Y),N)=g(X,SY)$ for all $X,Y\in\chi(\Sigma)$. The formulas of Gauss and Weingarten state, respectively, that for all $X,Y\in\chi(\Sigma)$

\begin{linenomath*}
\begin{equation}\label{e1}
\overline{\nabla}_XY=\nabla_XY+B(X,Y) \hspace{1cm} \textnormal{and} \hspace{1cm} \overline{\nabla}_XN=-SX.
\end{equation}
\end{linenomath*}

The Riemann curvature tensor $R$ of the surface $\Sigma$ is related to the second fundamental form and the curvature tensor $\overline{R}$ of the ambient space  by means of the Gauss equation
\begin{linenomath*}
\begin{equation}\label{ac1}
R(X,Y)Z=(\overline{R}(X,Y)Z)^T- g(SY,Z)SX+g(SX,Z)SY,
\end{equation}
\end{linenomath*}
for all $X,Y,Z\in\chi(\Sigma)$, here $X^T$ describes the tangent component of $X$ over the surface $\Sigma$.  The Codazzi equation is given by
\begin{linenomath*}
\begin{equation}\label{wc1}
\overline{R}(X,Y)N=\nabla_XSY-\nabla_YSX-S[X,Y],
\end{equation}
\end{linenomath*}
where $X,Y,Z\in\chi(\Sigma)$ and $N$ is the unit normal vector field compatible with the orientation of $\Sigma$.

In order to give an expression for  the Gauss and Codazzi equations for immersed surfaces in the warped product $\mrf$ we introduce some notations. Let $\varphi:\Sigma \looparrowright\mrf$ be an immersion from a complete, oriented surface into the warped product $\mrf$. Let $N$ be the  unit normal vector field to $\Sigma$ compatible with its orientation. We can decompose the vertical field $\xi$ in the tangent and normal components with respect to the surface $\Sigma$. That is, we can write
\begin{linenomath*}
\begin{equation}\label{e2}
\xi=T+\nu N
\end{equation}
\end{linenomath*}
where $T$ is the tangent projection of $\xi$ over $\Sigma$ and $\nu N$ is the orthogonal projection of $\xi$ over the unit normal vector field $N$.  The function $\nu$ is called \emph{ angle function}.

With the aid of the vector field $T\in\chi(\Sigma)$ we define a tensor $\mathcal{T}$ of order four over $\Sigma$ 
\begin{linenomath*}
$$\mathcal{T}:\chi(\Sigma)^4\to \R$$
\end{linenomath*}
given by
\begin{linenomath*}
$$ \mathcal{T}_{X}^Y(Z,W)=g(X,Z)g(Y,T)g(W,T)- g(X,W)g(Y,T)g(Z,T)$$
\end{linenomath*}
for vector fields $X,Y,Z,W\in\chi(\Sigma)$. 

With those definitions and notations we can state the following proposition.

\begin{proposition}\label{wd1}
Let $\Sigma$  be a simply connected Riemannian surface with metric $g$, Riemannian connection $\nabla$ and Riemann curvature tensor $R$.  Then,  for all $X,Y,Z,W\in\chi(\Sigma)$, we have 
\vspace{.3cm}

\begin{linenomath*}
{ \footnotesize 
\vspace{-0.2cm}
\begin{equation}\tag{\bf Gauss equation}
\label{gauss}\mathcal{R}(X,Y,Z,W)  =  g(R(X,Y)Z,W)- g(SX,Z)g(SY,W) +g(SX,W)g(SY,Z).
\end{equation}}
{\footnotesize
\vspace{-0.2cm} 
\begin{equation}\tag{\bf Codazzi equation}
\label{codazzi}\mathcal{S}(X,Y) =  \nabla_XSY-\nabla_YSX-S([X,Y]).
\end{equation}}
\end{linenomath*}

And also, the following equations are satisfied by the immersion

{\footnotesize 
\begin{linenomath*}
\begin{eqnarray}
\vert T\vert^2+\nu^2&=&1\\[10pt]
\nabla_XT-\nu SX & = & f^\prime[X-g(X,T)T] \\[15pt]
g(SX,T)+d\nu(X) & = & -f^\prime\, \nu\, g(X,T)
\end{eqnarray} 
\end{linenomath*}}
where
\begin{linenomath*}
\footnotesize{
\begin{eqnarray*}
\mathcal{R}(X,Y,Z,W) & = & \left((f^\prime)^2-\kappa e^{-2f}\right)[g(X,W)g(Y,Z)-g(X,Z)g(Y,W)]- \\
& & -\left(f^{\prime\prime}+\kappa e^{-2f}\right) [\mathcal{T}_X^Y(Z,W)-\mathcal{T}_Y^X(Z,W)]
\end{eqnarray*}
}
and
\begin{equation*}
\mathcal{S}(X,Y)=-\nu\left(f^{\prime\prime}+\kappa e^{-2f}\right)[g(X,T)Y-g(Y,T)X].
\end{equation*}
\end{linenomath*}
\end{proposition}

A proof for this proposition can be found in \cite{ortega2014fundamental}, where the authors proved that those equations are in fact the compatibility equations for isometric immersions in $\mrf$.


\section{Equations of totally umbilical surfaces invariant by isometries in \texorpdfstring{$\mrf$}{TEXT}}\label{examples}

In this section we assume that there is a totally umbilical surface invariant by one paramenter group of isometries of the ambient space and we find the equations that such a surface satisfies. These equations are important since in the proof of the Main Theorem we show that a totally umbilical surface satisfy some of these equations so, in particular, it must be contained in a totally umbilical surface invariant by one parameter group of isometries. 

The isometries of $\M(\kappa)$ are  extended naturally to isometries of $\mrf$. Therefore, there exists some positive isometries depending on the value of $\kappa$.

\begin{itemize}
\item[$\bullet$] If $\kappa=-1$, we have: rotational (elliptic), horizontal translation (parabolic) and hyperbolic isometries.  
\item[$\bullet$] If $\kappa=0$, we have:  translations with respect to a fixed vector and rotational isometries.  
\item[$\bullet$] If $\kappa=1$, we have: rotational isometry.  
\end{itemize}
Before describe such invariant surfaces we recall some necessary computation.
\begin{linenomath*}
\subsection{Covariant derivative of \texorpdfstring{$\mrf$}{TEXT}}

A global orthonormal frame on $\mrf$ is given by 
$$E_1=\dfrac{\partial_x}{\lambda e^f}, \hspace{.4cm} E_2=\dfrac{\partial_y}{\lambda e^f}, \hspace{.4cm} E_3=\xi,$$
\end{linenomath*}
with associated coframe
\begin{linenomath*}
$$\omega^1=\lambda e^fdx, \hspace{.4cm} \omega^2=\lambda e^f dy, \hspace{.4cm} \omega^3=d\xi.$$
\end{linenomath*}
Recall that the connections forms $\omega_j^i$ and the Riemannian connection are satisfies 

\begin{multicols}{2}
\begin{enumerate}
\item $d\omega^i=\Sigma_j \omega^j\Lambda \omega_j^i$, with  $\omega_j^k+\omega_k^j=0$ 
\item $\overline{\nabla}_X E_i=\Sigma_k \omega_j^k(X)E_k$.
\end{enumerate}
\end{multicols}
In our case, the non-zero connections form are given by 
\begin{linenomath*}
$$ \omega^1_2=\dfrac{\lambda_y}{\lambda^2 e^f}\omega^1-\dfrac{\lambda_x}{\lambda^2 e^f}\omega^2, \hspace{.4cm} \omega^1_3=f^\prime \omega^1, \hspace{.4cm} \omega^2_3=f^\prime\omega^2.$$
\end{linenomath*}
Denoting by $E_{ij}=\overline{\nabla}_{E_i}E_j$, the non-zero components of the Riemannian connection are given by
\begin{linenomath*}
$$
E_{11}= -\dfrac{\lambda_y}{\lambda^2 e^f} E_2-f^\prime E_3  \hspace{.4cm}  E_{12}=\dfrac{\lambda_y}{\lambda^2 e^f}E_1  \hspace{.4cm} E_{21}=\dfrac{\lambda_x}{\lambda^2 e^f} E_2 $$
$$
E_{22}= -\dfrac{\lambda_x}{\lambda^2 e^f} E_1-f^\prime E_3 \hspace{.4cm}  E_{13}=E_{31}=f^\prime E_1 \hspace{.4cm}  E_{23}=E_{32}=f^\prime  E_2.
$$
\end{linenomath*}
\subsection{Totally umbilical surfaces rotationally invariant}\label{secex1}
The idea to obtain a surface invariant by a one-parameter group of rotational isometries is simple. We will consider polar coordinates for the space $\M(\kappa)$, of the form
\begin{linenomath*}
\begin{equation}
\left\{
  \begin{array}{rcl}
    x &=& \rho_\kappa\cos\omega
     \vspace{.3cm} \\ 
   y &=& \rho_\kappa\sin\omega 
  \end{array}
\right.
\end{equation}
\end{linenomath*}
where, depending of $\kappa$:
\begin{itemize}
\item[$\bullet$] If $\kappa=-1$, we have $\rho_\kappa=\tanh\left(\dfrac{\rho}{2}\right)$. 
\item[$\bullet$] If $\kappa=0$, we have $\rho_\kappa=\rho$. 
\item[$\bullet$] If $\kappa=1$, we have $\rho_\kappa=\tan\left(\dfrac{\rho}{2}\right)$.   
\end{itemize}
for $\rho\in(0,+\infty)$ and $\theta\in[0,2\pi)$.

Also,  we consider the vertical plane $\Pi=\{(x,y,t)\in\mrf;y=0\}$ and a smooth curve $\alpha(s)$ parametrized by arclength in $\Pi$, that is
\begin{linenomath*}
$$\alpha_\kappa(s)=(\rho_\kappa(s),0,t(s)) \hspace{.5cm} \textnormal{and} \hspace{.5cm} e^{2f(s)}\rho_s^2+t_s^2=1$$
\end{linenomath*}
where $f(s)=f\vert_{\alpha_\kappa(s)}$. In the plane $\Pi$ we consider the unit normal $N$ to the curve $\alpha_\kappa$ so that the basis $(\alpha_\kappa^\prime(s),N(s))$ is positively oriented for each $s$. Now, we apply a one-parameter group of rotational movements around the axis $\{(0,0)_f\times\R\}$. So, the curve $\alpha_\kappa$ generates a rotationally-invariant surface which we denote by $\Sigma_\kappa$, we orient via $N$ the surface $\Sigma_\kappa$. Note that this surface can be parameterized by
\begin{linenomath*}
 \begin{equation*}
    \psi_\kappa(s,\omega)=(\rho_\kappa(s)\cos(\omega),\rho_\kappa(s)\sin(\omega),t(s)).
\end{equation*}
\end{linenomath*}
A standard computation give us the main curvatures $\kappa_1$ and $\kappa_2$ of the surface $\Sigma_\kappa$. More precisely, we obtain:

\begin{multicols}{2}
\begin{itemize}
\item[$\bullet$] $\kappa_1=\dfrac{e^f\left(t_{ss}\rho_s-2f_tt_s^2\rho_s-e^{2f}f_t\rho_s^3-t_s\rho_{ss}\right)}{(e^{2f(s)}\rho_s^2+t_s^2)^{\frac{3}{2}}}$ \vspace{.3cm}
\item[$\bullet$]  $\kappa_2=\dfrac{e^{-f}\left(-e^{2f}f_t\rho_s+z_\kappa(s)\right)}{(e^{2f(s)}\rho_s^2+t_s^2)^{\frac{1}{2}}},$ 
\end{itemize}
\end{multicols}
%
%

where, depending of value of $\kappa$:

\begin{multicols}{2}
\begin{itemize}
\item[$\bullet$] If $\kappa=-1$, we have $z_\kappa(s)=t_s\coth\rho$. \vspace{.3cm} 
\item[$\bullet$] If $\kappa=0$, we have $z_\kappa(s)=\dfrac{t_s}{\rho}$. \vspace{.3cm} 
\end{itemize}
\end{multicols}
\begin{center}
\begin{itemize}
\item[$\bullet$] If $\kappa=1$, we have $z_\kappa(s)=t_s\cot\rho$.   
\end{itemize}
\end{center}
The totally umbilical condition $\kappa_1=\kappa_2$ generates a second order differential equation whose first integral determines a system of two ordinary differential equations of first order of the form
\begin{linenomath*}
\begin{equation}\label{eqrot1}
\left\{
  \begin{array}{rcl}
    e^{2f(s)}\rho_s^2+t_s^2 &=& 1
     \vspace{.3cm} \\ 
   t_s -h_\kappa(s) &=& 0
  \end{array}
\right.
\end{equation}
\end{linenomath*}
where for each $\kappa$:
\begin{itemize}
\item[$\bullet$] If $\kappa=-1$, we have $h_\kappa(s)=c_0\sinh\rho$, $c_0>0$. \vspace{.3cm} 
\item[$\bullet$] If $\kappa=0$, we have $h_\kappa(s)=c_0\rho(s)$, $c_0>0$. \vspace{.3cm}
\item[$\bullet$] If $\kappa=1$, we have $h_\kappa(s)=c_0\sin\rho$, $c_0>0$.   
\end{itemize}
Conversely, any solution of the system of ODE's \eqref{eqrot1} with initial condition $(\rho(s_0), t(s_0)) = (\rho_0,t_0)$ determines a curve $\alpha_\kappa(s)$, with $\alpha_\kappa(s)=(\rho_\kappa(s),0,t(s))\subset\Pi$, which generates a rotationally-invariant totally umbilical surface $\Sigma_\kappa$ immersed into the warped product $\mrf$.

\subsection{ Totally umbilical surfaces invariant by translations}
There are three kinds of translations in $\mrf.$
\begin{itemize}
\item[$\bullet$] If $\kappa=0$, we have horizontal translations coming from a fixed vector.\vspace{.3cm} 
\item[$\bullet$] If $\kappa=-1$, we have parabolic and hyperbolic translations.
\end{itemize}
 The idea to construct surfaces invariant by such translations in $\mrf$ is the same as the previous case. We  consider a vertical plane $\Pi$ and a curve $\alpha\subset\Pi$ (the profile curve) parametrized by the arclength $s$. In order to simplify the computations, we consider three different vertical planes, depending of the kind of translation we are considering. More precisely.
 \begin{itemize}
 \item[$\bullet$] If $\kappa=0$, $\Pi_0=\{(x,y,t)\in\R_f\times\R;y=0\}$. \vspace{.3cm} 
 \item[$\bullet$] If $\kappa=-1$, we consider the half-plane model for the hyperbolic space. And \vspace{.3cm} 
 \begin{itemize}
 \item[$\ast$] $\Pi_1=\{(x,y,t)\in\mathbb{H}_f\times\R,x=0\}$, when we deal with parabolic translations.\vspace{.3cm} 
 \item[$\ast$] $\Pi_2=\{(x,y,t)\in\mathbb{H}_f\times\R,x^2+y^2=1,y>0\}$, when we deal with hyperbolic translations. In this case, we can parametrized such vertical plane, by using cylindrical coordinates, that is 
\begin{linenomath*}
 \begin{equation*}
\left\{
  \begin{array}{rcl}
    x &=& \cos\theta
     \vspace{.3cm} \\ 
    y &=& \sin\theta 
    \vspace{.3cm} \\ 
    t &=& t
  \end{array}
\right.
\end{equation*}
\end{linenomath*}
for $\theta\in(0,\pi)$ and $t\in\R$.
 \end{itemize}
 \end{itemize}

Now we consider a smooth curve $\alpha_i\subset\Pi_i$, $i=1,2,3$ parametrized by the arclength $s$. In the plane $\Pi_i$ we consider the unit vector field $N_i$ which is normal to the curve $\alpha_i$, so that the basis $(\alpha_i^\prime(s),N_i(s))$ is positively oriented for each $s$. Note that
\begin{itemize}
\item[$\bullet$] $\alpha_0(s)=(\rho(s),0,t(s)) \hspace{.5cm} \textnormal{and} \hspace{.5cm} e^{2f(s)}\rho_s^2+t_s^2=1$. \vspace{.3cm} 
\item[$\bullet$] $\alpha_1(s)=(0,\rho(s),t(s)) \hspace{.5cm} \textnormal{and} \hspace{.5cm} \dfrac{e^{2f(s)}}{\rho(s)^2}\rho_s^2+t_s^2=1$. \vspace{.3cm} 
\item[$\bullet$] $\alpha_2(s)=\left(\cos\rho(s),\sin\rho(s),t(s)\right) \hspace{.5cm} \textnormal{and} \hspace{.5cm} \dfrac{e^{2f(s)}}{\sin^2 \rho}\rho_s^2+t_s^2=1$.
\end{itemize}
being $f(s)=f|_{\alpha(s)}.$
\vspace{.3cm}

\indent Now, we apply a one-parameter group of translations. Thus, the curve $\alpha_i$ generates a surface invariant by such translations, which we denote by $\Sigma_i$, $i=0,1,2$. We orient by $N_i$ the surface $\Sigma_i$. Each surface $\Sigma_i$ can be parametrized by
\begin{itemize}
\item[$\bullet$] Here, we are considering a fixed direction $\beta\in\mathbb{S}^1$
\begin{linenomath*}
\begin{equation*}
    \psi_0(s,\omega)=(\rho(s)+\omega\cos(\beta),\omega\sin(\beta),t(s)).
\end{equation*}
\end{linenomath*}
For $(s,\omega)\in J\times\R$, here $J$ denotes an open interval. Up to an isometry of $\rf$, we can suppose that $\beta=\pi/2$. \vspace{.3cm} 
\item[$\bullet$] Recall that we are considering the half-plane model for the hyperbolic space
\begin{linenomath*}
\begin{equation*}
    \psi_1(s,\omega)=(\omega,\rho(s),t(s)).
\end{equation*}
\end{linenomath*}
For $(s,\omega)\in \R\times J$, here $J$ denotes an open interval. \vspace{.3cm} 
\item[$\bullet$] We are considering cylindrical coordinates
\begin{linenomath*}
\begin{equation*}
    \psi_2(s,\omega)=\left(\omega\cos\rho(s),\omega\sin\rho(s),t(s)\right),
\end{equation*}
\end{linenomath*}
where $(s,\omega)\in\mathbb{R}\times\mathbb{R}^+.$
\end{itemize}

Using the connections forms, a straightforward computation give us that the principal curvatures $\kappa_1^i$ and $\kappa_2^i$, $i=0,1,2$; computed with respect to the orientation $N_i$ are given by
\begin{linenomath*}
\begin{equation*}
\kappa_1^1=\kappa_1^2=-\kappa_1^0=\dfrac{e^f\lambda_i\left(2f_tt_s^2\rho_s+e^{2f}t_s\rho_s^3\lambda_i^2+t_s\rho_{ss}-t_{ss}\rho_s-t_s\rho_s^2 h_i\right)}{\left(t_s^2+e^{2f}\rho_s^2\lambda_i^2\right)^{\frac{3}{2}}}
\end{equation*}\end{linenomath*}
and
\begin{linenomath*}
\begin{equation*}
\kappa_2^i=\dfrac{e^{-f}t_sm_i+e^ff_t\rho_s\lambda_i}{\left(t_s^2+e^{2f}\rho_s^2\lambda_i^2\right)^{\frac{1}{2}}}
\end{equation*}\end{linenomath*}
where $\lambda_i(s)=\lambda(\alpha_i(s))$,
\begin{itemize}
\item[$\bullet$]$m_0(s)=0$ and $h_0(s)=0$. \vspace{.3cm}
\item[$\bullet$] $m_1(s)=1$ and $h_1(s)=\dfrac{1}{\rho(s)}$. \vspace{.3cm}
\item[$\bullet$]  $m_2(s)=\cos\rho(s)$ and $h_2(s)=\cot\rho(s)$. \vspace{.3cm}
\end{itemize}
The umbilicity condition $\kappa_1^i=\kappa_2^i$ becomes $t_{ss}=z_i(s)t_s$, where
\begin{multicols}{3}
\begin{itemize}
\item[$\bullet$] $z_0(s)=0$. \vspace{.3cm}
\item[$\bullet$] $z_1(s)=-\dfrac{\rho_s}{\rho}$. \vspace{.3cm}
\item[$\bullet$] $z_2(s)=-\cot(\rho)\rho_s$.\vspace{.3cm}
\end{itemize}
\end{multicols}
Which is a second order ordinary differential equation, whose first integral is given by $t_s=p_i(s)$ with
\begin{multicols}{3}
\begin{itemize}
\item[$\bullet$] $p_0(s)=c_0$, 
\item[$\bullet$] $p_1(s)=\dfrac{1}{c_0\rho}$, 
\item[$\bullet$] $p_2(s)=c_0\cosec(\rho)$,
\end{itemize}
\end{multicols}
for some positive constant $c_0$. Conversely, any solution of the ODE's system
\begin{linenomath*}
\begin{equation*}
\left\{
  \begin{array}{rcl}
    \vert| \gamma_i^\prime(s)\vert| &=& 1
     \vspace{.3cm} \\ 
   t_s &=& p_i(s) \hspace{.1cm} 
  \end{array}
\right.
\end{equation*}\end{linenomath*}
for some constant $c_o$, with initial condition at $t=t_0$ determines a curve $\alpha_i(s)\subset\Pi$, which generates a totally umbilical surface $\Sigma_i$ invariant by translations, immersed into the warped product $\mrf$.

Finally, we end this section with the following lemma.
 
\begin{lemma}\label{lis1}
Let $\Sigma\looparrowright\rf$ be a totally umbilical surface, invariant by horizontal translations, which is generated by the curve $\alpha(s)\subset\Pi=\{(x,y,t)\in\rf;x=0\}$, where $s$ is the arclength parameter of $\alpha$. Then, up to an isometry of the ambient space, we have
\begin{itemize}
\item[(a)] $\Sigma$ is a totally geodesic, vertical cylinder or,
\item[(b)] $\Sigma$ is a plane whose unit normal vector field $N$ make a constant angle $\gamma\in(0,\pi/2)$ with the vertical field $\xi$ which is tangent to $\R$. Furthermore, if $\gamma=0$, $\Sigma$ is the slice $\R^2_f\times\{t_0\}$ for some $t_0\in\R$. 
\end{itemize}   
\end{lemma}
\begin{proof}
The umbilicity condition is
\begin{linenomath*}
\begin{equation*}
t_{ss}=0
\end{equation*}\end{linenomath*}
which implies $t(s)=sin(\gamma) s+c_0$ for some fixed $(\gamma,c_0)\in\mathbb{S}^1\times\R$. Since $s$ is the arclength parameter of $\alpha$, the function $x$ satisfies $x_s=\cos(\gamma) e^{-f(s)}$. Note that, the profile curve $\alpha$ is given by
\begin{linenomath*}
\begin{equation}\label{pc1}
\alpha(s)=(x(s),0,t(s)) \hspace{.2cm}\textnormal{with} \hspace{.2cm} x_s=\cos(\gamma)e^{-f(s)} \hspace{.2cm} \textnormal{and} \hspace{.2cm} t(s)=\sin(\gamma)s+c_0.
\end{equation}\end{linenomath*}
This implies that the functions $x(s)$ and $t(s)$ are monotone for almost all $\gamma$. From here, we deduce the lemma. 
\end{proof}

\section{Main Theorem} \label{class}
In this section we give a proof of the main theorem, before that we  prove some lemmas. Recall the vertical vector field $\xi$ has been decomposed in the tangent part $T$ and the normal part $\nu N$, where $N$ is a unit normal vector field to $\Sigma$. Let $(u, v)$ be a local isothermal coordinates in a open set of $\Sigma$ around $p$, so the first fundamental form $g$, in
these coordinates is 
\begin{linenomath*}
 \begin{equation}\label{m1}
g=\dfrac{1}{\eta}(du^2+dv^2)
\end{equation}\end{linenomath*}
for some positive function $\eta=\eta(u,v)>0$. These coordinates are available on the interior of the set of umbilical points and also on a neighbourhood of non umbilical points. So, the set of points where the isothermal  coordinates $(u, v)$ are available is dense on $\Sigma$. Thus, properties obtained on this set are extended to $\Sigma$ by continuity. For a smooth function $h:\Sigma \to \R$, we denote by $\grad(h)$ the gradient of $h$ on $\Sigma$ and by $\grad_0(h)$ the vector field $(h_u,h_v)$ on $\Sigma$, that is $\grad_0(h)=(h_u,h_v)$. Under this notation we have the following lemmas.

\begin{lemma}\label{wl0}
Let $\Sigma$ be an immersed surface in $\mrf$ and suppose that $((u,v),\eta,U)$  is local isothermal coordinate around a fixed point $p\in\Sigma$. Let $h:\Sigma \to \R$ be a smooth real function on $\Sigma$. Then,
\begin{linenomath*}
\begin{equation}\label{grad0}
\eta \hspace{.1cm}\grad(h)=\grad_0(h)
\end{equation}\end{linenomath*}
with $\grad$ and $\grad_0$ defined as above.
\end{lemma}
\begin{proof}
The expression of the gradient of a function $h$ is  
\begin{linenomath*}
\begin{equation*}
\grad(h)=\sum_{i=1}^2 a_i\partial_{x_i} \hspace{.5cm} \textnormal{where} \hspace{.5cm}  a_i=\sum_{j=1}^2 g^{ij}\dfrac{\partial h}{\partial_{x_j}},  
\end{equation*}\end{linenomath*}
where $((x_1,x_2),g,U)$ is a local coordinate  around a fixed point $p\in\Sigma$,  $g=(g_{ij})$ is the induced metric on $\Sigma$ with inverse $(g^{ij})$.

So taking isothermal coordinates $(u,v)$ around $p\in\Sigma$ whose metric is written as $g=\dfrac{1}{\eta}(du^2+dv^2)$ we obtain the lemma.

\end{proof}
\begin{lemma}\label{wl1}
Let $\Sigma$ be an immersed oriented totally umbilical surface in $\mrf$ and  $(u,v)$  be a local isothermal coordinate around $p\in\Sigma$. Let $\varrho:\Sigma \to \R$ a umbilical function on $\Sigma$. Then,
\begin{linenomath*}
\begin{equation}\label{we3}
\eta \hspace{.1cm}\grad(\varrho)=\nu T (f^{\prime\prime} +\kappa e^{-2f})
\end{equation}\end{linenomath*}
where $\grad(\varrho)$ denotes the gradient of $\varrho$ in $\Sigma$, $f^{\prime\prime}$ denotes the second derivative of the warping function $f$ restrict to $\Sigma$ and $\nu$ is the angle function.
\end{lemma}
\begin{proof}
Locally $\Sigma$ is the image of an embedding $\varphi:\Omega \to \mrf$  where $\Omega$ is an open disk in $\M(\kappa)$. Call $(u, v)$ the isothermal coordinates on $\Omega$ and consider the unit normal field $N$ on $\varphi(\Omega)$. As $\Sigma$ is totally umbilical  there exists a umbilical function $\varrho:\Omega \to \R $ such that
\begin{linenomath*}
 $$\overline{\nabla}_{\varphi_u}N  =  \varrho \varphi_u \ \ \textrm{and} \ \ \overline{\nabla}_{\varphi_v}N  =  \varrho \varphi_v.$$
 \end{linenomath*}
This implies
\begin{linenomath*}
\begin{equation*}
\overline{\nabla}_{\varphi_v}(\overline{\nabla}_{\varphi_u}N)-\overline{\nabla}_{\varphi_u}(\overline{\nabla}_{\varphi_v}N)=\varrho_v\varphi_u-\varrho_u\varphi_v
\end{equation*}\end{linenomath*}
That is
\begin{linenomath*}
\begin{eqnarray}\label{we4}
\overline{R}(\varphi_u,\varphi_v)N=\varrho_v\varphi_u-\varrho_u\varphi_v &=&S(\varphi_v,\varphi_u)\\
\label{we5}  &= &\nu (f^{\prime\prime} +\kappa e^{-2f})[g(\varphi_v,T)\varphi_u-g(\varphi_u,T)\varphi_v],
\end{eqnarray}
\end{linenomath*}
where the last equality comes from the compatibility equations, and  $g$ denotes the metric on $\Sigma$. Comparing  equations \eqref{we4} and \eqref{we5} we obtain 
\begin{linenomath*}
\begin{equation*}
\grad_0(\varrho)=\nu T (f^{\prime\prime} +\kappa e^{-2f}).
\end{equation*}\end{linenomath*}
We conclude the proof using Lemma \ref{wl0}.
\end{proof}

From equation \eqref{we3}, when $\kappa=0$ and $f^{\prime\prime}=0$ we have $\varrho$ is constant and $\mrf$ is $\mathbb{R}^3$ or $\mrf$ is $\mathbb{H}^3$, depending on the choice of the warping function $f$. More generally,  the solution of the ODE 
\begin{linenomath*}
\begin{equation}\label{edo0}
f^{\prime\prime} +\kappa e^{-2f}=0
\end{equation}\end{linenomath*}
gives rise to warped products $\mrf$ where, if there exist totally umbilical surfaces $\Sigma$, then the umbilical function $\varrho$ is constant. In order to solve the EDO \eqref{edo0} we consider the smooth function $f(t)=\ln F(t)$. The following theorem gives a precise characterization of such warped spaces.
\begin{theorem}\label{edol1}
Let  $F_1(t)$ and $F_2(t)$ be  smooth functions given by:
\begin{linenomath*}
\begin{equation*}
F_1(t) =  \dfrac{e^{-\sqrt{c_0}t}}{4c_0 c}\left[c^2e^{2\sqrt{c_0}t}-2c_0\kappa\right] \hspace{.3cm} \textnormal{and} \hspace{.3cm} F_2(t) = \dfrac{e^{-\sqrt{c_0}t}}{4c_0 c}\left[c^2 -2c_0e^{2\sqrt{c_0}}\kappa\right],
\end{equation*}\end{linenomath*}
where $c_0$ and $c$ are positive constants. Then, $f_i(t)=\ln(F_i(t))$, $i=1,2$, give rise to the warped products $\M(\kappa)_{f_i}\times\R$ in which if there exist totally umbilical surfaces $\Sigma$, the associated umbilical function $\varrho$ is constant.
\end{theorem}

\begin{proof}
The function $f_i$ solves the ODE \eqref{edol1}. Then by Lemma \ref{wl1} the conclusion follows. 
\end{proof}

The case where the umbilical function is constant is well understood so, from now on, we assume $(f^{\prime\prime} +\kappa e^{-2f})\neq0$.

\begin{theorem}\label{t1} 
Let $\Sigma$ be a totally umbilical surface immersed in the warped product  $\mrf$. Then $\Sigma$
is part of a totally umbilical surface, which is invariant by a one-parameter group of isometries of $\mrf$. More precisely, up to an ambient isometry, $\Sigma$ is contained in one of the examples described in Section \ref{examples}.
\end{theorem}
\begin{proof}
Let $N$ be a  unit normal vector field to the surface $\Sigma$ and assume that $\Sigma$ is oriented by $N$. Locally $\Sigma$ is the image of an embedding $\varphi:\Omega \to \mrf$  where $\Omega$ is an open disk in $\R^2$. Let $((u,v),\eta,\Omega)$  be  a local isothermal coordinates around a fixed point $p\in\Sigma$. Since $\varphi$ is totally umbilical, there exists a function $\varrho:\Omega \to \R$ such that $\overline{\nabla}_vN=\varrho v$ for any $v$ tangent to $\Sigma$. From Lemma \ref{wl1}, we have
\begin{linenomath*}
\begin{equation*}
\eta \hspace{.1cm}\grad(\varrho)=\nu T (f^{\prime\prime} +\kappa e^{-2f})
\end{equation*}\end{linenomath*}
where $\nu$ is the angle function. There exists two cases for $\varrho$ depending on the existence of critical points. Therefore, we will separate the proof in two steps. 

\item[\textbf{Case 1.}] First assume that $\varrho$ has no critical point. In particular each level curve of $\varrho$ is orthogonal to $T$, moreover, since $0=\langle\gamma^\prime, \xi\rangle=\langle\gamma^\prime, T\rangle$, the level curves are horizontal, that is, belongs to some $\M(\kappa)\times\{t_0\}$.

In order to prove that the totally umbilical surface $\Sigma$ is an invariant surface, we study  the horizontal and the vertical curves of $\Sigma$.
\begin{claim}\label{claim1}
Set $\Sigma_t=\Sigma\cap(\M(\kappa)\times\{t\})$ and let $\gamma:I \to \Sigma_t$ be a regular parametrization of $\Sigma_t$. Then
\begin{itemize}
\item[(i)] For each $t$, the geodesic curvature of the  horizontal curve $\gamma(I)$ in $\Sigma_t$ is constant.
\item[(ii)] The angle between $\Sigma$ and $\M\times\{t\}$ is constant along $\Sigma_t$ for each $t$. 
\end{itemize}
\end{claim}
Suppose that $\gamma:I \to \Sigma_t$ is  arclength parametrized by  $s$ and denote  by $\gamma^\prime(s)$ the derivative of $\gamma$ with respect to $s$ at the point $s\in I$ (which is horizontal), we have 
\begin{linenomath*}
\begin{equation*}
\dfrac{d\nu}{ds}=\dfrac{d}{ds}g(N,\xi)=\overline{g}(\overline{\nabla}_{\gamma^\prime}N,\xi)+\overline{g}(N,\overline{\nabla}_{\gamma^\prime}\xi)= \overline{g}(\varrho\gamma^\prime,\xi) + \overline{g}(N,\xi(f) \gamma^\prime)=0
\end{equation*}\end{linenomath*}
here we use the item $(1)$ from Lemma \ref{l1} and that $\Sigma$ is totally umbilical. Therefore $\nu$ is constant along $\gamma$. 

Now, let  $n$ be the unit normal vector field in $T(\M(\kappa)\times\{t_0\})$ along $\gamma$ with the orientation induced by $N$. Let $\theta$ be the oriented angle between $n$ and $N$, hence 
\begin{linenomath*}
\begin{equation*}
N(\gamma(s))=\cos(\theta(s)) n(\gamma(s))+ \sin(\theta(s)) \xi(\gamma(s)).
\end{equation*}\end{linenomath*}
Then, $\sin(\theta(s))=g(N(\gamma(s)), \xi(\gamma(s)))=g(N(\gamma(s)), T(\gamma(s))+\nu N(\gamma(s)))=\nu$. So, in particular, since $\nu$ is constant  $\theta$ is constant along $\gamma$. On the other hand
\begin{linenomath*}
\begin{eqnarray*}
\varrho(\gamma(s)) & = & \overline{g}(\overline{\nabla}_{\gamma^\prime(s)}N,\gamma^\prime(s)) \\
& = & \cos\theta\overline{g}(\overline{\nabla}_{\gamma^\prime(s)}n,\gamma^\prime(s))+ \sin\theta\overline{g}(\overline{\nabla}_{\gamma^\prime(s)}\xi,\gamma^\prime(s)) \\
& = & \cos\theta\overline{g}(\overline{\nabla}_{\gamma^\prime(s)}n,\gamma^\prime(s))+ \sin\theta f^\prime(\gamma(s))\vert \gamma^\prime(s)\vert
\end{eqnarray*}
\end{linenomath*}
Observe that $\overline{g}(\overline{\nabla}_{\gamma^\prime(s)}n,\gamma^\prime(s))$ is the geodesic curvature of $\gamma$ in $\M(\kappa)_f\times\{t_0\}$. Since $\varrho$ and $\theta$ are constants along $\gamma$, we deduce that $\gamma$ has constant geodesic curvature, which proves the claim.
\vspace{.3cm}

\begin{claim}\label{claim2}
Let $\Sigma$ be a  totally umbilical surface. Assume that the angle function $\nu$ does not vanish anywhere. Let $c:J\subset\mathbb{R}\longmapsto \Sigma$ be a curvature line associated to the vector field $T$. Then, $c(J)$ is contained in a vertical totally geodesic plane.
\end{claim}

\begin{proof}(Proof of Claim \ref{claim2}) Let $c:J\subset\R \to \Sigma$ be a curvature line associated to the field $T$. Thus $c^\prime(u)=T(c(u))$, where $c^\prime(\cdot)$ denotes the derivative of $c$ with respect to the variable $u\in J$. We want to show that $c(J)$ is contained in a vertical totally geodesic plane.  We denote by $c_h(u)=\pi_1(c(u)), u\in J$, where $\pi_1$ is the projection on the first factor.

\begin{remark}
For example, when $\nu\neq0$ (our first case), it is equivalent to show that $c_h$  is a geodesic. 
\end{remark}

Since $\nu\neq 0$, the curve $c$ does not have a vertical point, then $c^\prime_h$ does not vanish. In this case, it is sufficient to show that $\nabla^2_{c_h^\prime}c_h^\prime$ is always parallel to $c_h^\prime$, here $\nabla^2$ denotes the riemannian connection of $\M(\kappa)                                                                                                                                                                                                                                                                                                                                                     \times\{0\}$.

As $c^\prime=c_h^\prime+(1-\nu^2)\xi$, we have
\begin{linenomath*}
$$
\begin{array}{cllll}
\overline{\nabla}_TT & = & \overline{\nabla}_{c^\prime}c^\prime & = & \overline{\nabla}_{c^\prime}[c_h^\prime+(1-\nu^2)\xi] \\
   &  &  & = & \overline{\nabla}_{c_h^\prime}c_h^\prime +\dfrac{d}{du}(1-\nu^2)\xi +2(1-\nu^2)\overline{\nabla}_{c_h^\prime}\xi \\
 \hspace{1.5cm}  & & &=& \nabla^2_{c_h^\prime}c_h^\prime -2\nu\nu_u\xi-\overline{g}(c^\prime_h,c^\prime_h)f^\prime\xi -2(1-\nu^2)f^\prime c_h^\prime  \ \textnormal{(Using Lemma \ref{l1})}
\end{array}
$$
\end{linenomath*}
Since $T$ is a principal direction, $\overline{\nabla}_TN=\varrho T$ and
\begin{linenomath*}
\begin{equation*}
\nu_u=\dfrac{d}{du}\overline{g}(N,\xi)=\overline{g}(\overline{\nabla}_TN,\xi)+\overline{g}(N,\overline{\nabla}_T\xi)=\overline{g}(\varrho T,\xi)+\overline{g}(N,\overline{\nabla}_T\xi)
\end{equation*}\end{linenomath*}
setting $h(u)=\overline{g}(N,\overline{\nabla}_T\xi)$, we have
\begin{linenomath*}
\begin{equation}\label{we6}
\nu_u=\varrho(1-\nu^2) +h(u).
\end{equation}\end{linenomath*}
Which implies
\begin{linenomath*}
\begin{equation}\label{we7}
\overline{\nabla}_TT = \nabla^2_{c_h^\prime}c_h^\prime -2\nu[\varrho(1-\nu^2) +h(u)]\xi-\overline{g}(c^\prime_h,c^\prime_h)f^\prime\xi -2(1-\nu^2)f^\prime c_h^\prime.
\end{equation}\end{linenomath*}
On the other hand,
\begin{linenomath*}
\begin{eqnarray*}
\overline{\nabla}_TT & = & \overline{\nabla}_{T}(\xi-\nu N) 
   =  \overline{\nabla}_{T}\xi-\nu_u N-\nu\varrho T 
  =  \overline{\nabla}_{T}\xi +\nu_u\left(\dfrac{\xi-T}{\nu}\right) -\nu\varrho[c_h^\prime+(1-\nu^2)\xi] \\
 &= & \overline{\nabla}_{T}\xi -\dfrac{\nu_u}{\nu}[\xi-c_h^\prime-(1-\nu^2)\xi] -\nu\varrho[c_h^\prime+(1-\nu^2)\xi] \\
 & = & \overline{\nabla}_{T}\xi+ \left[\dfrac{\nu_u}{\nu}(1-\nu^2)-\dfrac{\nu_u}{\nu}-\varrho\nu(1-\nu^2) \right]\xi +\left[\dfrac{\nu_u}{\nu}-\nu\varrho\right]c_h^\prime \\
 & = & \overline{\nabla}_{T}\xi +\left[\dfrac{\varrho(1-\nu^2)^2}{\nu}-\dfrac{\varrho(1-\nu^2)}{\nu}-\varrho\nu(1-\nu^2)\right]\xi + \left[\dfrac{1-\nu^2}{\nu}h(u)-\dfrac{h(u)}{\nu}\right]\xi +\left[\dfrac{\nu_u}{\nu}-\nu\varrho\right]c_h^\prime  \\
 & = & \overline{\nabla}_{T}\xi -\nu h(u)\xi+\left[\dfrac{\nu_u}{\nu}-\nu\varrho\right]c_h^\prime -2\varrho\nu(1-\nu^2)\xi.
\end{eqnarray*}
\end{linenomath*}
Using Lemma \ref{l1}, we obtain 
\begin{linenomath*}
\begin{equation}\label{we8}
\overline{\nabla}_TT = f^\prime c_h^\prime -\nu h(u)\xi+\left[\dfrac{\nu_u}{\nu}-\nu\varrho\right]c_h^\prime -2\lambda\nu(1-\nu^2)\xi.
\end{equation}\end{linenomath*} 
Then, equations \eqref{we7} and \eqref{we8} imply
\begin{linenomath*}
\begin{equation}\label{eqa1}
\nabla^2_{c_h^\prime}c_h^\prime =\left( f^\prime+\left(\dfrac{\nu_u}{\nu}-\nu h\right)+2(1-\nu^2)f^\prime\right)c^\prime_h+\left(\nu h_u+\overline{g}(c^\prime_h, c^\prime_h)f^\prime\right)\xi.
\end{equation}\end{linenomath*}
In order to finish the proof of this claim, we have to show that $\left(\nu h_u+\overline{g}(c^\prime_h, c^\prime_h)f^\prime\right)=0$. In fact,

\begin{linenomath*}
\begin{eqnarray*}
\nu h_u+\overline{g}(c^\prime_h, c^\prime_h)f^\prime & = & -\nu\overline{g}(N,\overline{\nabla}_{c^\prime}\xi) -\overline{g}(c_h^\prime ,c_h^\prime)f^\prime  =  \overline{g}(-\nu N,\overline{\nabla}_{c^\prime}\xi) -\overline{g}(c_h^\prime ,c_h^\prime)f^\prime 
=  \overline{g}(T-\xi,\overline{\nabla}_{c^\prime}\xi) -\overline{g}(c_h^\prime ,c_h^\prime)f^\prime \\
   & = &  \overline{g}(T,\overline{\nabla}_{c^\prime}\xi) -\overline{g}(c_h^\prime ,c_h^\prime)f^\prime 
    =  \overline{g}(T,\overline{\nabla}_{c^\prime_h}\xi) -\overline{g}(c_h^\prime ,c_h^\prime)f^\prime 
    =   \overline{g}(T,f^\prime c^\prime_h) -\overline{g}(c_h^\prime ,c_h^\prime)f^\prime \\
   & = &  \overline{g}(c^\prime_h+(1-\nu^2)\xi,f^\prime c^\prime_h) -\overline{g}(c_h^\prime ,c_h^\prime)f^\prime =   \overline{g}(c^\prime_h, c^\prime_h)f^\prime -\overline{g}(c_h^\prime ,c_h^\prime)f^\prime =  0.
\end{eqnarray*}
\end{linenomath*}
So, \eqref{eqa1} becames
\begin{linenomath*}
\begin{equation*}
\nabla^2_{c_h^\prime}c_h^\prime=\left[\dfrac{\nu_u}{\nu}-\nu\lambda+f^\prime+2(1-\nu^2)f^\prime\right]c_h^\prime,
\end{equation*}\end{linenomath*}
 it means that $c_h(J)$ is a geodesic in $\M(\kappa)_f\times\{0\}$ and therefore $c(J)$ is contained in a vertical totally geodesic plane as claimed.
 \end{proof}
\vspace{.3cm}
Now we are ready to prove that $\Sigma$ is an invariant surface. Let us consider the horizontal curve $\gamma: I \to \Sigma_t$ arclength parametrized. Let $s_1, s_2\in I$ and call $c:(-\epsilon,\epsilon) \to \Sigma$ the integral curve of $T$ such that $c(0)=\gamma(s_1)$ and $\overline{c}:(-\epsilon,\epsilon) \to \Sigma$ the integral curve of $T$ such that $\overline{c}(0)=\gamma(s_2)$. Let us call $c_3$ (resp. $\overline{c}_3$) the third coordinate of $c$ (resp. $\overline{c}$) and let $u$ be the parameter in $(-\epsilon,\epsilon)$, we have 
\begin{linenomath*}
\begin{equation*}
c_3^\prime(u)=\overline{g}(c^\prime(u),\xi) = \overline{g}(T(c(u)),\xi)= 1-\nu^2(c(u))= 1-\nu^2(c_3(u)),
\end{equation*}\end{linenomath*}
where in the last equality, by abuse of notation, we write $\nu^2(c(u))=\nu^2(c_3(u))$ since the Claim \ref{claim1} assures that the angle function depends only on the third coordinate  of points in $\Sigma$. 

Since  $c(0)=\gamma(s_1),\  \overline{c}(0)=\gamma(s_2)$ with $\gamma\subset\Sigma_t$,  $c_3$ and $\overline{c}_3$ verify the same first order differential equation with the same initial condition at $u=0$. We conclude that $c_3=\overline{c}_3$.

Claim \ref{claim2} says that the curves $c$ and $\overline{c}$ are contained in the totally geodesic vertical planes $P$ and $\overline{P}$, respectively.                                                                                                                                                                                                                                                                                                                                                                                                                                                                                                                                                                                                                                                                                                                                                                                                                                                                                                                                                                                                                                                                                                                                                                                                                                                                                                                                                                                                                                                                                                                                                                                                                                                                                                                                                                                                                                                                                                                                                                                                                                                                                                                                                                                                                                                                                                                                                                                                                                                                                                                                                                                                                                                                                                                                                                                                                                                                                                                                                                                                                                                                                                                                                                                                                                                               

Now, we show that such totally umbilical surface is invariant by isometries. Let us consider the curve $\Gamma\subset\M(\kappa)_f\times\{0\}$ which is complete and have constant geodesic curvature, $\Gamma$ is obtained by extending $\gamma$, that is $\gamma\subset\Gamma$. Notice that there exists a positive isometry $\varphi$ of the space form $\M(\kappa)$ such that $\varphi(\Gamma)=\Gamma$, $\varphi(c(0))=\overline{c}(0)$ and preserving the orientation of $\Gamma$. Therefore the isometry of the ambient space, defined by $F(p,t)=(\varphi(p),t)$ of $\M(\kappa)_f\times\R$ sends $P$ to $\overline{P}$. Such isometry shows that the curves $\overline{c}$ and $F\circ c$, in the vertical plane, $\overline{P}$ have the same vertical component and make the same angle with the horizontal for each $u\in(-\epsilon,\epsilon)$. Since such curves have the same horizontal e vertical components, we deduce that these curves coincide, that is $F\circ c=\overline{c}$. This complete the proof of the theorem for the first case. 
\begin{remark}
Suppose that $\nu$ vanishes on an open interval $J_1\subset J$. In this case, we want to show that a part of  $\Sigma$ is a vertical cylinder. We denote by $W$ a unit horizontal field along $c$ which is tangent to $\Sigma$ and for each $u\in J$, we let $\{P(u)\}_{u\in J_1}$ be the family of vertical totally geodesic planes such that, $P(u)$ containing $c(u)$ and is orthogonal at $c(u)$ to $W(u)$. Let $u_0\in J_1$ and notice that along the horizontal curve of $\Sigma$ which cut $c(J_1)$ through $c(u_0)$, the vector field $N$ is horizontal (since the angle between $\Sigma$ and $\M(\kappa)\times\{t\}$ is constant). This means that an open set of $\Sigma$ including $c(J_1)$ is part of a cylinder $\gamma_f\times\R$, where $\gamma$ is some horizontal curve. This implies that $W$ is constant along $J_1$, and thus so is $P$.
\end{remark}


\item[\textbf{Case 2.}]  Suppose that $\varrho$ has some critical points. 
Let $W\subset\Sigma$ be a connected component (if any) of the interior of the set of critical points of $\lambda$. By equation \eqref{we3}, since we are assuming that we have $(f^{\prime\prime}+\kappa e^{-2f})\neq0$, the product $\nu T=0.$ This implies that $T\equiv0$ or  $\nu\equiv0$ in $W$.  In the former case, $W$ is part of a slice $\M(\kappa)\times\{t_0\}$ and in the latter case $W$ is part of a cylinder  $\Gamma\times\R$, where $\Gamma$ is some curve in $\M(\kappa)\times\{0\}$. Since $\Sigma$ is totally umbilic, $\Gamma$ has to be a geodesic an so $W$ is totally geodesic.

Let now $V\subset\Sigma$ be a connected component (if any)  of the set of regular points of $\varrho$. From the first part of the proof, we know that $V$
 is part of one the invariant examples given in Section \ref{examples}. 
 If there were both no-empty sets $W$ and $V$ contained in $\Sigma$,  $\Sigma$ would be a gluing of  pieces of totally geodesic surfaces and pieces of the invariant examples constructed in Section \ref{examples}. But one analysis of those surfaces implies that there is no a smooth surface like this. It means that $\Sigma$ is a totally geodesic  or it is contained in one of the one-parameter invariant surfaces described in Section \ref{examples}.  This conclude the proof of the theorem.
\end{proof}

\section{A non-trivial example}\label{ntex}

It is clear that to find an explicit solution for the profile curve $\alpha(s)$ which generate an invariant, totally umbilical surfaces $\Sigma$ immersed into the warped product $\mrf$, we need to know the warping function $f$, we mean, we need to know which is the ambient space. In this section we focus our attention to the construction of such surface by considering a non-trivial warping function $f$. More precisely, for a given smooth function $f$, we consider $\Sigma\looparrowright\mrf$ a totally umbilical, invariant surface, immersed into the warped product $\mrf$. Thus, we have the following lemma.

\begin{lemma}\label{exrotle}
Let $\mrf$ be the warped product with $\kappa\ge0$ and consider a curve $\alpha(s)\subset\Pi\subset\mrf$ parametrized by the arclength $s$. Then, for the fixed warping function $f$, we have:
\begin{enumerate}
\item If $f(t)=\ln\left(\dfrac{1}{\sin(t)}\right)$ and $\kappa=0$, the curve $\alpha(s)=(\rho(s),0,t(s))$ for smooth functions $\rho=\rho(s)$ and $t=t(s)$ given by
\begin{linenomath*}
\begin{equation*}
t(s) = 2\arctan(e^s) \hspace{.3cm} \textnormal{and} \hspace{.3cm} \rho(s) = \dfrac{1}{\cos(s)} \hspace{.3cm} \textnormal{with} \hspace{.3cm} s\in(0,\pi/2)
\end{equation*}\end{linenomath*}
generates a rotationally-invariant, totally umbilical surface $\Sigma\subset\rf$ embedded into the warped product $\rf$.
\item If $f(t)=\ln\left(\dfrac{\cos(t)}{\sqrt{\sin(t)}}\right)$ and $\kappa=1$, the curve $\alpha(s)=(\rho(s),0,t(s))$ for smooth functions $\rho=\rho(s)$ and $t=t(s)$ given by
\begin{linenomath*}
\begin{equation*}
\rho(s) = t(s) = 2\arctan(e^s) \hspace{.3cm} \textnormal{with} \hspace{.3cm} s\in(0,\pi/2)
\end{equation*}\end{linenomath*}
generates a rotationally-invariant, totally umbilical surface $\Sigma\subset\sfr$ immersed into the warped product $\sfr$.
\end{enumerate} 
\end{lemma}
\begin{proof}
In order o simplify the computations, we prove the item $(1)$, the item $(2)$ is similar. Consider $\Sigma\looparrowright\rf$ a totally umbilical, rotationally-invariant surface, immersed into the warped product $\rf$. We consider, as in section \eqref{secex1},  the vertical plane
\begin{linenomath*}
$$\Pi=\{(x,y,t)\in\rf;y=0\}$$ 
\end{linenomath*}
and the smooth curve $\alpha=\alpha(s)\subset\Pi$ which generates the rotationally-invariant surface $\Sigma$. From the umbilicity equation \eqref{eqrot1} we have
\begin{linenomath*}
\begin{equation}\label{eR1}
c_0^2\dfrac{t_{ss}^2}{\sin^2t}+t_s^2=1 \hspace{.5cm} \textnormal{and} \hspace{.5cm} t_s=c_0\hspace{.1cm}\rho(s)
\end{equation}\end{linenomath*}  
for some constant $c_0\in\R$. Notice that, the problem of finding rotationally-invariant surfaces immersed into the warped product $\rf$  consists in determine the admissible expressions of the profile curve $\alpha(s)$, we mean, we wish to find, for each $c_0\in\R$, the possible integral curve of the ODE's system \eqref{eR1}. Therefore, in order to find some example of such surface, we consider without loss of generality $c_0=1$.

From the first equation in \eqref{eR1}, we can suppose that there exist a smooth function $\beta=\beta(s)$ such that $t_s=\cos(\beta)$, thus $t_{ss}=\cos(\beta)\beta_s$. Thus such first equation becomes 
\begin{linenomath*}
$$\left(\dfrac{\beta_s}{\sin(t)}\right)^2=1, \hspace{.5cm} \textnormal{that is} \hspace{.5cm} \dfrac{\beta_s}{\sin(t)}=\pm1.$$
\end{linenomath*}
If we consider the positive expression, the ODE's system \eqref{eR1} can be written as
\begin{linenomath*}
\begin{equation}\label{eR3}
\left\{
  \begin{array}{rcl}
    t_s &=& \sin(\beta)
     \vspace{.3cm} \\ 
    \beta_s &=& \sin(t)
  \end{array}
\right.
\end{equation}\end{linenomath*}
which is equivalent to
\begin{linenomath*}
\begin{equation*}
\dfrac{t_s}{\sin(\beta)}=\dfrac{\beta_s}{\sin(t)}, \hspace{.5cm} \textnormal{that is} \hspace{.5cm} (\cos(t))_s=(\cos(\beta))_s.
\end{equation*}\end{linenomath*}
By integrating this equation we obtain
\begin{linenomath*}
\begin{equation}\label{eR2}
\cos(t)=\cos(\beta)+\cos(c)=2\cos\left(\dfrac{\beta+c}{2}\right)\cos\left(\dfrac{\beta-c}{2}\right)
\end{equation}\end{linenomath*} 
for some constant $c\in\R$. Choosing the constant as being $c=0$, we obtain $t=\beta$, from the ODE's system \eqref{eR3} we obtain $t_s=\sin(t)$, whose solution is given by
\begin{linenomath*}
\begin{equation}\label{eR4}
t(s)=2\arctan(e^{s+c_1})
\end{equation}\end{linenomath*}  
for some constant $c_1\in\R$. If we suppose that
$$\lim_{s\to 0}t(s)=0 \ \ \textrm{then} \ \ t(s)=2\arctan(e^{s}),$$
which completes the proof.
\end{proof}


\section{Appendix: Geometric behaviour of curves and cylinders}

In the Appendix we describe mean, extrinsic  and intrinsic curvatures of a vertical cylinder in $\mrf$. Also we study geodesics of $\mrf$.

\subsection{Cylinders in \texorpdfstring{$\mrf$}{TEXT}}
 
In this section we state some geometric properties of vertical cylinders immersed in $\mrf$. We start recalling the geometric behaviour of curves in a space endowed with conformal metrics.

\begin{definition}
Two Riemannian metrics $\sigma$ and $\overline{\sigma}$ on a surface $\Sigma$ are said to be conformal if there exists a smooth real function $\varphi\in C^\infty(\Sigma)$ such that
\begin{linenomath*}
$$\overline{g}=e^\varphi g.$$
\end{linenomath*}
\end{definition}

If we denote by $\nabla^\sigma$ and $\nabla^{\overline{\sigma}}$ the Levi-Civita connection with  respect to $\sigma$ and $\overline{\sigma}$, then for any vector fields $X,Y,Z\in\chi(\Sigma)$, we have
\begin{linenomath*}
$$\nabla^{\overline{\sigma}}_XY=\nabla^\sigma_XY+X(\varphi)Y+Y(\varphi)X -\sigma(X,Y) \grad_{\sigma}(\varphi).$$
\end{linenomath*}
\begin{lemma}\label{comet}
Let $\Sigma$ be a surface which can be  endowed with two conformal metrics $\sigma$ and $\overline{\sigma}=e^{2\varphi} \sigma$. Consider a curve $\gamma:I\subset\R \to \Sigma$. We denote by $\kappa_{\sigma}$ and $\kappa_{\overline{\sigma}}$ the geodesic curvature of $\gamma$ with respect to the metric $\sigma$ and $\overline{\sigma}$, respectively. Then
\begin{linenomath*}
$$\kappa_{\overline{\sigma}}=e^{-\varphi}\left(\kappa_{\sigma}-\dfrac{\partial \varphi}{\partial\eta}\right),$$
\end{linenomath*}
where $\dfrac{\partial \varphi}{\partial\eta}$  denotes the directional derivative with respect to the inner unit normal vector $\eta$ of $\gamma$ (with respect of the metric $\sigma$).
\end{lemma}
\begin{proof}
The geodesic curvature $\kappa_{\sigma}$ is calculated by 
$\kappa_{\sigma}=\sigma(\nabla^{\sigma}_vv,n),$
where $v$ is the unit tangent vector field to $\gamma$ and $n$ the inner unit normal vector field to $\gamma$ with respect to $\sigma$ (the unit vector field pointing to the bounded region if the curve is a Jordan curve). 

For the metric $\overline{\sigma}$ we choose the arclength-parameter $s$ to $\gamma$, that is $\overline{\sigma}(\gamma_s(s),\gamma_s(s))=1,$
here $\gamma_s$ denote the derivative of $\gamma $ with respect to $s$. Note that this is equivalent to $\sigma(e^\varphi\gamma_s(s),e^\varphi\gamma_s(s))=1.$ Now, consider the inner unit normal vector $\eta$ of $\gamma$ with respect to the metric $\sigma$, that is $\sigma(\eta,\eta)=1,$
or equivalently $\overline{\sigma}(e^{-\varphi}\eta,e^{-\varphi}\eta)=1.$
Then
\begin{linenomath*}
\begin{eqnarray*}
\kappa_{\overline{\sigma}} & = & \overline{\sigma}\left(\nabla^{\overline{\sigma}}_{\gamma_s}\gamma_s,e^{-\varphi}\eta\right) 
=  \overline{\sigma}\left(\nabla^{\sigma}_{\gamma_s}\gamma_s +2\sigma(\grad_{\sigma}\varphi,\gamma_s)\gamma_s -\sigma(\gamma_s,\gamma_s)\grad_{\sigma}\varphi,e^{-\varphi}\eta\right) \\
           & = & e^{2\varphi}\sigma\left(\nabla^{\sigma}_{\gamma_s}\gamma_s -\sigma(\gamma_s,\gamma_s)\grad_{\sigma}\varphi,e^{-\varphi}\eta\right)
           =  e^{\varphi}\sigma\left(\nabla^{\sigma}_{\gamma_s}\gamma_s,\eta\right) -e^{-\varphi}\sigma\left(\grad_{\sigma}\varphi,\eta\right) \\
           & = & e^{\varphi}\sigma\left(e^{-2\varphi}\nabla^{\sigma}_{e^\varphi\gamma_s}(e^\varphi\gamma_s)-e^{-\varphi}\sigma(\gamma_s,\gamma_s)\gamma_s,\eta\right) -e^{-\varphi}\sigma\left(\grad_{\sigma}\varphi,\eta\right) \\
            & = & e^{-\varphi}\sigma\left(\nabla^{\sigma}_{e^\varphi\gamma_s}(e^\varphi\gamma_s),\eta\right) -e^{-\varphi}\sigma\left(\grad_{\sigma}\varphi,\eta\right) \\
            & = & e^{-\varphi}\left(\kappa_{\sigma}-\sigma(\grad_{\sigma}\varphi,\eta)\right)
\end{eqnarray*}
\end{linenomath*}
where $\grad_{\sigma}$ denotes the gradient with respect to the metric $\sigma$, which conclude the proof.
\end{proof}
Observe that, if $e^\varphi=\lambda>0$, then $ \varphi = \ln\lambda \ \ \textrm{and} \ \ \kappa_{\overline{\sigma}} = \dfrac{1}{\lambda}\left(\kappa_\sigma -\dfrac{1}{\lambda}\dfrac{\partial\lambda}{\partial\eta}\right).$
%

Now, we focus our attention in the study of some geometric properties of vertical cylinders. For a fixed $t=t_0$, we denote by $\M_{t_0}(\kappa)=\M(\kappa)\times\{t_0\}$, a copy of $\M(\kappa)$ at height $t_0$.  Let $\gamma\subset \M(\kappa)$ be a complete simple curve with geodesic curvature $\kappa_{g_2}$, $g_2$ the metric on $\M(\kappa)$,  and $\kappa^0$ the geodesic curvature of $\gamma\subset\R^2$, that is, as a curve in the Euclidean space $\R^2$ (when $\kappa=0$ we have $\kappa_{g_2}=\kappa^0$). We also consider the complete surface $Q^\gamma=\pi^{-1}_1(\gamma)$, which is called the (vertical) cylinder over $\gamma$.

\begin{lemma}\label{Ca3}  Let $Q^\gamma\subset\mrf$ be the cylinder generated by a complete curve $\gamma\subset\M(\kappa)$. For each $t_0$ denote by $\gamma_{t_0}$ the intersection between $Q$ and $\M(\kappa)_{t_0}$, that is $\gamma_{t_0}=Q^\gamma\cap\M_{t_0}(\kappa)$. Then we have.
\begin{enumerate}
\item The first and second fundamental forms $(I,II)$ of $Q^\gamma$ are given by
$$ I = e^{2f(t)}ds^2+dt^2 \ \ \textrm{and} \ \ II =  -e^{f(t)}\kappa_{g_2} ds^2.$$

%
In particular, the main curvatures $\kappa_1$ and $\kappa_2$ are given by
\begin{linenomath*}
$$\kappa_1(p,t)=-e^{-f(t)}\kappa_{g_2}(p) \hspace{.3cm} \textnormal{and} \hspace{.3cm} \kappa_2(p,t)=0,\ \ \ (p,t)\in\mrf.$$
\end{linenomath*}
Thus $Q^\gamma$ is a totally geodesic surface if and only if $\gamma$ is a geodesic in $\M(\kappa)$.
\item The mean curvature $H(p,t)$ and the extrinsic curvature $K_e(p,t)$ in  $(p,t)\in Q^\gamma$ are given by 
\begin{linenomath*}
$$ 2H(p,t) = -e^{-f(t)}\kappa_{g_2}(p), \ \ K_e(p,t) =  0.$$
\end{linenomath*}
%
%
In particular $Q^\gamma$ is a minimal surface if,  and only if, the $\gamma$ is a geodesic in $\M(\kappa)$.
\item  The intrinsic curvature $K_i(p,t)$ of $Q^\gamma$ is
\begin{linenomath*}
$$K_i(p,t)=-(f^{\prime}(t))^2-f^{\prime\prime}(t).$$
\end{linenomath*}
In particular, the linear function $f(t)=at+b$ for $a\neq0$ and $b$ real numbers, give rise to a complete embedding of the hyperbolic space $\h$.
\end{enumerate}

\end{lemma}
\begin{proof}

We consider the curve $\gamma(s)=(x(s),y(s))\subset\M(k)$ parametrized by the arclength parameter $s\in I$, where $I$ is an interval of $\R$. That is $g_{2}(\gamma^\prime,\gamma^\prime)=1$. Also we consider the unit normal vector field (with respect to the Euclidean metric $g_2^0$), $\eta=-\lambda\, y^\prime\,\partial_x+\lambda \,x^\prime\,\partial_y$, where $x^\prime$ denotes the derivative of $x$ with respect to $s$.

We consider the following parametrization of the cylinder $Q^\gamma$ 
\begin{linenomath*}
$$\Phi(s,t)=(x(s),y(s),t),$$
\end{linenomath*}
where $(s,t)\in I\times\R$. Consider the orthonormal frame $\{E_1,E_2,\xi\}$ where
$$E_1=\lambda^{-1}e^{-f}\partial_x \hspace{.3cm} \textnormal{and} \hspace{.3cm} E_2=\lambda^{-1}e^{-f}\partial_y $$
A natural orthogonal frame on $Q^\gamma$ is given by
\begin{linenomath*}
\begin{equation}
  \Phi_s =\lambda e^f x^\prime(s)E_{1}+\lambda e^f y^\prime(s)E_{2}, \hspace{1cm}
  \Phi_t = \xi, \hspace{0.5cm} \textnormal{and} \hspace{.5cm}
  N = \lambda y^\prime(s)E_{1}-\lambda x^\prime(s)E_{2}.
\end{equation}
\end{linenomath*}
 We denote by $E,F,G$ the coefficients of the first fundamental form and by by $e,f,g$ the coefficients of the second fundamental form of $Q$, we have
\begin{linenomath*}
$$ \begin{array}{lcc}
  E=e^{2f}\hspace{.3cm} & F=0 \hspace{.3cm}& G=1 \\
  e=e^{f(t)}\dfrac{1}{\lambda}\left[-\kappa^0+\dfrac{1}{\lambda}\dfrac{\partial\lambda}{\partial\eta}\right] \hspace{.3cm} & f=0\hspace{.3cm}  &  g=0.
 \end{array}$$
 \end{linenomath*}
This gives (1).  Using Lemma \ref{comet} we obtain (2). From the Gauss equation, see Proposition \ref{wd1}, we have
\begin{linenomath*}
\begin{equation*}
    K_i=detS-[(f^\prime)^2-\kappa e^{-2f}]-[f^{\prime\prime}+\kappa e^{-2f}]\vert T\vert^2
\end{equation*}\end{linenomath*}
since $\nu=0$, and $K_{e}=\det S$, we obtain (3).
\end{proof}

\subsection{Some curves in \texorpdfstring{$\mrf$}{TEXT}} Let $\Sigma$ be an immersed surface in $\mrf$. In this section we study the curvature lines of $\Sigma$ and also geodesics of the ambient space. 
   
\subsubsection{Geodesics in $\mrf$}
Let $\gamma$ be a geodesic of $\mrf$  parametrized by the arclength $s$. We decompose the tangent vector field $\gamma^\prime$ in the horizontal and vertical part, more precisely
\begin{linenomath*}
 $$\gamma'(s)=\gamma'_h(s)+\nu\xi,$$ 
\end{linenomath*}
 where $\nu=\bar{g}(\gamma'(s),\xi).$ For a fixed $t=t_0$, we denote the Riemannian connection of $\M_{t_0}(\kappa)$ by $\nabla^2$ and for sake of simplicity we omit the parameter $s$ in the calculations.


\begin{linenomath*}
\begin{eqnarray*} 
\overline{\nabla}_{\gamma^\prime}\gamma^\prime & = & \overline{\nabla}_{\gamma^\prime}\gamma^\prime_h+\nu^\prime\xi+\nu\overline{\nabla}_{\gamma^\prime}\xi 
= \overline{\nabla}_{\gamma^\prime_h}\gamma^\prime_h+ \nu\overline{\nabla}_{\xi}\gamma^\prime_h+\nu^\prime\xi+\nu\overline{\nabla}_{\gamma^\prime_h}\xi+\nu^2\overline{\nabla}_{\xi}\xi \\
& = & \nabla^2_{\gamma^\prime_h}\gamma^\prime_h + \bar{g}\left(\overline{\nabla}_{\gamma^\prime_h}\gamma^\prime_h,\xi\right)+ 2\nu f^\prime\gamma^\prime_h+\nu^\prime\xi 
= \nabla^2_{\gamma^\prime_h}\gamma^\prime_h - \bar{g}\left(\gamma^\prime_h,\overline{\nabla}_{\gamma^\prime_h}\xi\right)+ 2\nu f^\prime\gamma^\prime_h+\nu^\prime\xi \\
& = & \nabla^2_{\gamma^\prime_h}\gamma^\prime_h - f^\prime\bar{g}\left(\gamma^\prime_h,\gamma^\prime_h\right)+ 2\nu f^\prime\gamma^\prime_h+\nu^\prime\xi \\
& = & \nabla^2_{\gamma^\prime_h}\gamma^\prime_h+2\nu f^\prime\gamma^\prime_h - \left(f^\prime(1-\nu^2) -\nu^\prime\right)\xi.
\end{eqnarray*}
\end{linenomath*}
Since $\gamma(s)$ is a geodesic, we have $\overline{\nabla}_{\gamma^\prime}\gamma^\prime=0$. Therefore, this computation proves the following lemma.

\begin{lemma}
Let $\gamma:I\longrightarrow\mrf$ be a geodesic of $\mrf$. Then the projection of $\gamma$ over the slice $\M_{t_0}(\kappa)$ is (part of) a horizontal geodesic and  its angle function $\nu$ is given by
\begin{linenomath*}
$$\nu(s)=\tanh(f(t(s))+C),$$ for a  $C\in\mathbb{R}.$
\end{linenomath*}
\end{lemma}

 \subsubsection{Line of curvature}
Let $\Sigma\looparrowright\mrf$ be an isometrically immersed surface into the warped product $\mrf$. A connected curve $C\subset\Sigma$ is a \emph{curvature line} if for all $p\in C$ the tangent vector of $C$ at $p$ is a principal direction of $\Sigma$. As a consequence from Lemma \ref{l1} we have the next lemma.
\begin{lemma}\label{Lcl}
Let $\Sigma$ be an orientable immersed surface in $\mrf$, oriented by a unit normal vector field $N$. Suppose that the angle between $\Sigma$ and the slice $\M(\kappa)_{t}(\kappa)$ is constant along $\Sigma_t=\Sigma\cap\M_{t}(\kappa)$. Then, $\Sigma_t$ is a curvature line.
\end{lemma} 
 \begin{proof}
 On the one hand
\begin{linenomath*}
\begin{equation*}
0=\dfrac{d}{ds}\overline{g}(N,\xi)=\overline{g}(\overline{\nabla}_{\gamma^\prime(s)}N,\xi) +\overline{g}(N,\overline{\nabla}_{\gamma^\prime(s)}\xi)=\overline{g}(\overline{\nabla}_{\gamma^\prime(s)}N,T).
\end{equation*}\end{linenomath*}
On the other hand
\begin{linenomath*}
\begin{equation*}
0=\overline{g}(\gamma^\prime,N)=\overline{g}\left(\gamma^\prime,\dfrac{T-\xi}{\nu}\right)=\dfrac{1}{\nu}\overline{g}(\gamma^\prime,T),
\end{equation*}\end{linenomath*}
so,  $T$ is a principal direction on $\Sigma$ which means that $\Sigma_t$ is a curvature line.
 \end{proof}
\begin{remark}
Notice that any curve in a totally umbilical surface is a curvature line.
\end{remark}

\bibliographystyle{abbrv}
\bibliography{Ref}
%
%
%
\end{document}